\documentclass{amsart}
\author{Julien Melleray}
\address{Universit\'e Claude Bernard -- Lyon 1 \\
  Institut Camille Jordan, CNRS UMR 5208 \\
  43 boulevard du 11 novembre 1918 \\
  69622 Villeurbanne Cedex \\
  France}
\usepackage{amssymb}
\usepackage{amsmath}
\usepackage[initials,shortalphabetic]{amsrefs}
\usepackage[all]{xy}
\usepackage{mathpazo}
\usepackage{hyperref}
\usepackage{my-macros}
\numberwithin{equation}{section}

\title{Dynamical simplices and Fra\"iss\'e theory}

\begin{document}
\maketitle
\begin{abstract}
We simplify a criterion (due to Ibarluc\'ia and the author) which characterizes dynamical simplices, that is, sets $K$ of probability measures on a Cantor space $X$ for which there exists a minimal homeomorphism of $X$ whose set of invariant measures coincides with $K$. 
We then point out that this criterion is related to Fra\"iss\'e theory, and use that connection to provide a new proof of Downarowicz' theorem stating that any nonempty metrizable Choquet simplex is affinely homeomorphic to a dynamical simplex. The construction enables us to prove that there exist minimal homeomorphisms of a Cantor space which are speedup equivalent but not orbit equivalent, answering a question of D. Ash.

\end{abstract}

\section{Introduction}
In this paper, we continue investigations initiated in \cite{Ibarlucia2016} concerning \emph{dynamical simplices}, that is, sets $K$ of probability measures on a Cantor space $X$ such that there exists a minimal (i.e.~such that all orbits are dense) homeomorphism of $X$ whose set of invariant measures coincides with $K$. Dynamical simplices are natural invariants of orbit equivalence and, in fact, a famous theorem of Giordano--Putnam--Skau \cite{Giordano1995} asserts that they are complete invariants of orbit equivalence for minimal homeomorphisms (see section \ref{s:speedups} for details on orbit equivalence, speedups, and the Giordano--Putnam--Skau theorem). The main theorem of \cite{Ibarlucia2016} is the following.

\begin{theorem*}[Ibarluc\'ia--Melleray \cite{Ibarlucia2016}]
Let $X$ be a Cantor space. A set $K \subset P(X)$ is a dynamical simplex if and only if:
\begin{enumerate}
\item $K$ is compact and convex.
\item All elements of $K$ are atomless and have full support.
\item $K$ satisfies the \emph{Glasner--Weiss condition}: whenever $A,B$ are clopen subsets of $X$ and $\mu(A)< \mu(B)$ for all $\mu \in K$, there exists a clopen $C \subseteq B$ such that $\mu(A)=\mu(C)$ for all $\mu \in K$.
\item $K$ is \emph{approximately divisible}: for any clopen $A$, any integer $n$, and any $\varepsilon >0$, there exists a clopen $B \subset A$ such that $\mu(A)-\varepsilon \le n \mu(B) \le \mu(A)$  
\end{enumerate}
\end{theorem*}

The first two conditions are obviously necessary for $K$ to be the set of invariant measures for a minimal action of any group by homeomorphisms of $X$; the fact that the third one is necessary follows from a theorem of Glasner--Weiss \cite{Glasner1995a}, hence the terminology we adopt here. When \cite{Ibarlucia2016} was completed, the status of approximate divisibility was more ambiguous: on the one hand, it played a key part in the arguments; on the other hand, this assumption seems rather technical, and it is not hard to see that when $K$ has finitely many extreme points approximate divisibility is a consequence of the three other conditions in the above theorem. Thus it was asked in \cite{Ibarlucia2016} whether this assumption is really necessary; we prove here that it is in fact redundant, thus simplifying the characterization of a dynamical simplex and giving it its final form.

\begin{theorem*}
Let $X$ be a Cantor space. Assume that $K \subset P(X)$ is compact and convex; all elements of $K$ are atomless and have full support; and $K$ satisfies the Glasner--Weiss condition. Then $K$ is approximately divisible (hence $K$ is a dynamical simplex). 
\end{theorem*}

This theorem is obtained as a corollary of the following result, which is of independent interest.

\begin{theorem*}
Let $X$ be a Cantor space, and $G$ be a group of homeomorphisms of $X$ such that each $G$-orbit is dense. Then the set of all $G$-invariant Borel probability measures on $X$ is approximately divisible.
\end{theorem*}

The proof of this result is similar in spirit to some of the arguments used in \cite{Ibarlucia2016}. A more novel aspect of the work presented here is that we exploit a connection between dynamical simplices and Fra\"iss\'e theory, which enables us to build interesting examples. Along those lines, we obtain a new and rather elementary proof of a well-known theorem of Downarowicz.

\begin{theorem*}[Downarowicz \cite{Downarowicz1991}]
For any nonempty metrizable Choquet simplex $K$, there exists a minimal homeomorphism $\varphi$ of a Cantor space $X$ such that $K$ is affinely homeomorphic to the set of all $\varphi$-invariant Borel probability measures on $X$. 
\end{theorem*}

An interesting aspect of the construction used to prove this result is its flexibility; we exploit this to prove the following theorem.

\begin{theorem*}
There exist two dynamical simplices $K,L$ of a Cantor space $X$, and homeomorphisms $g,h \in \Homeo(X)$ such that $g_*K \subset L$ and $h_*L \subset K$, yet there is no $f \in \Homeo(X)$ such that $f_*K=L$.
\end{theorem*}

This answers a question recently raised by Ash \cite{Ash2016} and shows that the relation of speedup equivalence is strictly coarser than the relation of orbit equivalence. 

The paper is organized as follows: we first prove that approximate divisibility is redundant in the characterization of dynamical simplices. Then we outline some basics of Fra\"iss\'e theory and explain why this theory is relevant to the study of dynamical simplices. We exploit that connection to prove Downarowicz's theorem on realizability of Choquet simplices as dynamical simplices, and then adapt this construction to show that there exist minimal homeomorphisms which are speedup equivalent but not orbit equivalent. \\

\noindent {\bf Acknowledgements.} The research presented here was partially supported by ANR project GAMME (ANR-14-CE25-0004). I am grateful to T. Ibarluc\'ia for valuable corrections, comments and suggestions about the first draft of this paper. I thank T. Giordano for providing a much-needed bibliographical reference. An anonymous referee's comments and suggestions hopefully helped improve the paper's readability.

\section{Approximate divisibility of simplices of invariant measures}
Throughout this section $X$ stands for a Cantor space, and $P(X)$ denotes the space of Borel probability measures on $X$, with its usual topology (induced by the maps $\mu \mapsto \mu(A)$ as $A$ runs over all clopen subsets of $X$). 
Our objective in this section is to prove the following theorem, and then deduce from it that approximate divisibility is redundant in the characterization of a dynamical simplex.

\begin{theorem}\label{t:invariant}
Let  $G$ be a group of homeomorphisms of a Cantor space $X$ such that each $G$-orbit is dense. Then the set $K_G$ of all $G$-invariant Borel probability measures on $X$ is approximately divisible.
\end{theorem}

For the remainder of this section, we fix $G$ as above.

\begin{defn}
Fix an integer $N$, and a clopen set $A$. An \emph{$N$-dividing partition} of $A$ is a finite clopen partition $(U_{i,j})_{i \in I, j \in \{0,\ldots,n_i\}}$ of $A$ such that $n_i \ge N$ for all $i$ and there exists $g^i_1,\ldots,g^i_{n_i}$ in $G$ such that 
$g^i_{j}(U_{i,0})=U_{i,j}$ for all $i$ and all $j \in \{1,\ldots,n_i\}$. 

When $\mathcal U$ is an $N$-dividing partition of $A$, we also say that $\mathcal U$ covers $A$.
\end{defn}
By analogy with Kakutani--Rokhlin partitions we say that  $\{U_{i,j} \colon 0 \le j \le n_i\}$ is a column of the partition, with base $U_{i,0}$ (though the actual ordering of the partition does not matter here). Note that it follows from the definition of an $N$-dividing partition that $\mu(U_{i,j})=\mu(U_{i,0})$ for all $i$ and $j$ and all $\mu \in K_G$.

There are two operations on $N$-dividing partitions which will be useful to us. Given an $N$-dividing partition $\mcU$, a column $\mcC$ of $\mcU$ with base $U_{i,0}$, and a clopen partition $V_1,\ldots,V_k$ of $U_{i,0}$, we can form a new partition by replacing $\mcC$ with columns $(V_1,g^{i}_1(V_1),\ldots,g^{i}_{n_{i}}(V_1)),\ldots,(V_k,g^{i}_1(V_k),\ldots,g^{i}_{n_{i}}(V_k))$. The other operation is that, when $\mcU$ covers a clopen set $A$, $V$ is a clopen subset disjoint from $A$, $U_{i,j}$ is an atom of $\mcU$ and there exists $g \in G$ such that $g(U_{i,j})=V$, then we may extend $\mcU$ to an $N$-dividing partition that covers $A \cup V$ by setting $U_{i,n_i+1}=V$ and $g^i_{n_i+1}=g g^i_j$; we say that we have added $V$ on top of the $i$-th column of $\mcU$. 

If $\mcU'$ is obtained from $\mcU$ by applying these two operations finitely many times, we say that $\mcU'$ \emph{refines} $\mcU$. 

\begin{lemma}\label{l:inductivestep1}
Let $\mathcal U= (U_{i,j})_{i \in I, j \in \{0,\ldots,n_i\}}$ be an $N$-dividing partition of some clopen $A$, and $(V_0,h_1(V_0),\ldots,h_m(V_0))$
be an $N$-dividing partition of some clopen set $B$. Then there exists an $N$-dividing partition $\mathcal U'$ which refines $\mathcal U$ and covers a clopen $A' \supseteq A$, and $W_0 \subseteq V_0$ which is disjoint from $A'$, such that the union of $A'$ and $W_0,\ldots,h_m(W_0)$ is equal to $A \cup B$.
\end{lemma}
Note that the only difficulty is that $A$ and $B$ may not be disjoint, and the two $N$-dividing partitions may overlap, requiring a bit of care. 

\begin{proof}
As above we denote by $(g^i_j)$ some elements of $G$ witnessing that $\mathcal U$ is an $N$-dividing partition. 
Let $V=V_0 \cap A$; if $V=\emptyset$ we have nothing to do as $W_0=V_0$ works. Otherwise we set $W_0=V_0 \setminus V$. We then let $\mathcal P$ denote the clopen partition of $V$ generated by $V \cap h_1^{-1}(A), \ldots, V \cap h_m^{-1}(A)$. Note that for any atom $C$ of $\mathcal P$ and all $j$, $h_j(C)$ is either contained in $A$ or disjoint from it. Then, using the maps $g^i_j$, we may refine $\mathcal U$ (by partitioning the base of each column and replicating the corresponding partition along the column via $g^i_j$) to form a new $N$-dividing partition $\mathcal V$ which refines $\mcU$ and is finer than $\mathcal P \cup \{A \setminus V\}$. This implies that, for any atom $C$ of $\mathcal V$, either $C$ is contained in $V$ or is disjoint from $V$; and if $C$ is contained in $V$ then each of $C,h_1(C),\ldots, h_m(C)$ is either contained in $A$ or disjoint from it.

Given one of the columns of $\mathcal V$ which meets $V$, with base $V_{i,0}$, list the atoms $C_1,\ldots,C_k$ in this column that are contained in $V$; for each $l \in \{1,\ldots,k\} $ let $J_l \subseteq \{1,\ldots,m\}$ denote the set of all indices $j$ such that $h_j(C_l) \cap A = \emptyset$. We add each $h_j(C_l)$ for $j \in J_l$ on top of our column (if $J_l$ is empty the colum is not modified), which is fine since these sets are pairwise disjoint, are also disjoint from all elements of $\mathcal V$, and there is an element of $G$ mapping $V_{i,0}$ onto $h_j(C_l)$.

Once we have done this for all the columns of $\mathcal V$ which intersect $V$, we have produced an $N$-dividing partition $\mathcal U'$ which refines $\mcU$ and covers $A'=A \cup V \bigcup_{i\ge 1} h_i(V)$. Since $V$ as well as each $h_i(V)$ are disjoint from $W_0$, and $W_0$ is disjoint from $A$, we see that $A'$ and $W_0$ are disjoint, as required. 
\end{proof}

\begin{lemma}\label{l:inductivestep}
Fix an integer $N$, and assume that $A$, $B$ are clopen subsets such that there exists an $N$-dividing partition of $A$ and an $N$-dividing partition of $B$. Then there is an $N$-dividing partition of $A \cup B$.
\end{lemma}

\begin{proof}
By induction on the number of columns of the $N$-dividing partition of $B$, it is enough to consider the case where it is of the form $(V_0,h_1(V_0),\ldots,h_m(V_0))$ for some clopen $V_0$ and some $m \ge N$. Using the construction of Lemma \ref{l:inductivestep1} $m$ times, we produce $N$-dividing partitions $\mathcal U_i$ which refine $U$, and clopen sets $W_i$ contained in $V_0$ such that
\begin{itemize}
\item For all $i < m$, $\mathcal U_{i+1}$ refines $\mathcal U_i$ and $W_{i+1} \subseteq W_i$;
\item For all $i\le m$ the elements of $\mathcal U_{i}$ do not intersect $W_i, h_1(W_i),\ldots,h_i(W_i)$;
\item For all $i \le m$ the union of $\mathcal U_i$ and $(W_i,h_1(W_i),\ldots,h_m(W_i))$ covers $A \cup B$.
\end{itemize}
In the end, $\mcU_m$ is disjoint from $W_m,h_1(W_m),\ldots,h_m(W_m)$, so we may just add $(W_m,h_1(W_m),\ldots,h_m(W_m))$ as a new column to $\mcU_m$ to obtain the desired $N$-dividing partition of $A \cup B$.
\end{proof}

\begin{prop}
For any integer $N$ and any clopen $A$ there exists an $N$-dividing partition of $A$.
\end{prop}

\begin{proof}
Since  by assumption the $G$-orbit of any $x \in A$ is dense, its intersection with $A$ is infinite and we may find some clopen subset $U_0 \ni x$ and elements $g_1,\ldots,g_N$ of $G$ such that $U_0,g_1(U_0), \ldots,g_N(U_0)$ are disjoint and contained in $A$. Thus there exists a covering of $A$ by clopen sets which can each be covered by an $N$-dividing partition, and by compactness there exists such a covering which is finite. Then we obtain the desired $N$-dividing partition of $A$ from Lemma \ref{l:inductivestep}.
\end{proof}

\begin{proof}[Proof of Theorem \ref{t:invariant}]
Fix a nonempty clopen subset $A$ of $X$, an integer $n$ and $\varepsilon >0$. 
Then pick an $N$-dividing partition of $A$ for some $N > \frac{n}{\varepsilon}$; let us denote it as before by $(U_{i,j})_{i \in I, j \in \{0,\ldots,n_i\}}$ .For each $i$, we write $n_i+1= n p_i + q_i$ for some $q_i \in \{0,\ldots,n-1\}$ and then define, for all $k \in \{0,\ldots,n-1\}$,
$$B_k = \bigsqcup_{i \in I} \bigsqcup_{l=0}^{p_i-1} U_{i,k+nl}\ .  $$
Since the definition of an $N$-dividing partition implies that $\mu(U_{i,j})=\mu(U_{i,0})$ for all $i,j$ and $\mu \in K_G$, we have 
 $\mu(B_0)=\mu(B_1)=\ldots = \mu(B_{n-1})$ for all $\mu \in K_G$. Hence $n \mu(B_0) \le \mu(A)$ for all $\mu \in K_G$ since $B_0,\ldots,B_{n-1}$ are disjoint and contained in $A$. Also by construction, for any $\mu \in K_G$ we have 
$$\mu(A \setminus \bigsqcup_{k=0}^{n-1} B_k) \le \sum_{i \in I} q_i \mu(U_{i,0}) \le (n-1) \mu(\bigcup_{i \in I} U_{i,0})$$
In any $N$-dividing partition of $A$, the union of the bases of the columns must have $\mu$-measure less that $\frac{\mu(A)}{N}$ for all $\mu \in K_G$, so what we just obtained implies that
$$\mu(A) - n \mu(B_0) \le \frac{n-1}{N} \mu(A) \le \varepsilon \ . $$
Thus, $B=B_0$ is such that $\mu(A) - \varepsilon \le n \mu(B) \le \mu(A)$, proving that $K_G$ is approximately divisible.
\end{proof}

We are now ready to prove that the assumption of approximate divisibility in the definition of a dynamical simplex is redundant.

\begin{cor} \label{t:whydivide}
Assume that $K \subset P(X)$ is compact and convex, all elements of $K$ are atomless and have full support, and $K$ satisfies the Glasner--Weiss condition. Then $K$ is approximately divisible. 
\end{cor}

\begin{proof}
We pick $K$ as above, and let $H=\{h \in \Homeo(X) \colon \forall \mu \in K \ h_*\mu=\mu\}$. We first note that the argument used in the proof of Proposition 2.6 in \cite{Glasner1995a} shows that the action of $H$ on $X$ is transitive. Hence it follows from Theorem \ref{t:invariant} that the simplex $K_H$ of all $H$-invariant Borel probability measures on $X$ is approximately divisible. Since $K$ is contained in $K_H$, $K$ is also approximately divisible (we note that it then follows from the arguments of \cite{Ibarlucia2016} that $K=K_H$ but this is not needed here).
\end{proof}

\section{Dynamical simplices as Fra\"iss\'e limits}

In this section, we develop the connection between dynamical simplices and Fra\"iss\'e theory. We introduce the basic concepts of Fra\"iss\'e theory in our specific context. Since we do not assume familarity with this theory, we give some background details; we refer the reader to \cite{Hodges1993a} for a more thorough discussion.

For the moment, we fix a nonempty set $E$ (later, $E$ will be a Choquet simplex, and that additional structure on $E$ will play an important role), and the objects we consider are Boolean algebras endowed with a family of finitely additive probability measures $(\mu_e)_{e \in E}$. 
 We let $\mathcal B_E$ denote the class of all these objects: an element of $\mcB_E$ is of the form $(A,(\mu_e)_{e \in E})$ where $A$ is a Boolean algebra and each $\mu_e$ is a finitely additive probability measure on $A$. Since there should be no risk of confusion, we simply write $A \in \mcB_E$ and use the same letter $A$ also to denote the underlying Boolean algebra of the structure we are considering.

Allow us to point out, for those who are used to Fra\"iss\'e theory, that elements of $\mcB_E$ really are first order structures in disguise, in a language where each $\mu_e$ is coded by unary predicates $\mu_e^r$ which are interpreted via $\mu_e(a)=r \leftrightarrow a \in \mu_e^r$; this is for instance done in \cite{Kechris2007} (with just one measure in the language but this makes no essential difference). We adopt somewhat unusual conventions in the hope of improving readability in our particular case; one needs to pay attention to the fact that our language is uncountable, and so is the class $\mcB_E$.

Note that inside $\mcB_E$ we have natural notions of isomorphism, embedding, and substructure. Below, we call elements of $\mcB_E$ \emph{$E$-structures}.

\begin{defn}
Let $\mcK$ be a subclass of $\mcB_E$ . One says that $\mcK$ satisfies :
\begin{enumerate}
\item the \emph{hereditary property} if whenever $A \in \mcK$ and $B$ embeds in $A$ then also $B \in \mcK$.
\item the \emph{joint embedding property} if for any $A, B \in \mcK$ there exists $C \in \mcK$ such that both $A$ and $B$ embed in $C$.
\item the \emph{amalgamation property} if, for any $A,B,C \in \mcK$ and embeddings $\alpha \colon A \to B$, $\beta \colon A \to C$ there exists $D \in \mcK$ and embeddings $\alpha' \colon B \to D$, $\beta' \colon C \to D$ such that $\alpha' \circ \alpha = \beta' \circ \beta$.
\end{enumerate}
\end{defn}

All the classes we will consider contain the trivial boolean algebra $\{0,1\}$, with $\mu_e(0)=0$ and $\mu_e(1)=1$ for all $e$, which is a substructure of all elements of $\mcB_E$. Thus the joint embedding property will be implied by the amalgamation property. The hereditary property is usually easy to check; the amalgamation property is typically much trickier, and is intimately related to \emph{homogeneity}.

\begin{defn}
Given $A \in \mcB_E$, we say that $A$ is \emph{homogeneous} if any isomorphism between finite substructures of $A$ extends to an isomorphism of $A$. 
\end{defn}
Here we recall that an isomorphism between $E$-structures is an isomorphism of the underlying Boolean algebras which preserves the values of each $\mu_e$.

Homogeneous structures appear often in mathematics: for instance, $\Q$ with its usual ordering is homogeneous, since any order-preserving map between two finite subsets of $\Q$ extends to an order-preserving bijection. The countable atomless Boolean algebra is another good example to have in mind. 

\begin{defn}
Let $A \in \mcB_E$. The \emph{age} of $A$ is the class of all elements of $\mcB_E$ which are isomorphic to a finite substructure of $A$.
\end{defn}

Note that the age of any $E$-structure satisfies the hereditary property and the joint embedding property.

\begin{prop}\label{p:amalgamation}
Assume that $M \in \mcB_E$ is homogeneous. Then its age satisfies the amalgamation property.
\end{prop}

\begin{proof}
Let $A,B,C$ belong to the age of $M$ and $\alpha \colon A \to B$ and $\beta \colon A \to C$ be embeddings. Given the definition of an age, we may assume that $B,C \subset M$. The isomorphism $\alpha \beta^{-1} \colon \beta(A) \to \alpha (A)$ extends to an automorphism $f$ of $M$ by homogeneity. Let $D$ be the substructure of $M$ generated by $B$ and $f(C)$ (which is finite since a finitely generated Boolean algebra is finite), and $\alpha'$ the inclusion map from $B$ to $D$, $\beta'$ the restriction of $f$ to $C$. Then for all $a \in A$ we have $\beta' \beta (a)= \alpha \beta^{-1} \beta(a) = \alpha(a)= \alpha' \alpha(a)$.
\end{proof}

There exists a form of converse to the previous result, which applies to countable $E$-structures (that is, the underlying Boolean algebra is countable).

\begin{defn}
A subclass $\mcK$ of $\mcB_E$ is a \emph{Fra\"iss\'e class} if :
\begin{itemize}
\item All elements of $\mcK$ are finite.
\item $\mcK$ satisfies the hereditary property, the joint embedding property and the amalgamation property.
\item  $\mcK$ contains countably many structures up to isomorphism (we will sometimes say that $\mcK$ is countable).
\end{itemize}
\end{defn}

We already know that the age of a countable, homogeneous structure is a Fra\"iss\'e class. Note that $\mcB_E$ itself is not a Fra\"iss\'e class since it is not countable. 

To any Fra\"iss\'e class, a unique homogeneous structure is associated.

\begin{theorem}[Fra\"iss\'e]\label{t:Fraisse}
let $\mcK$ be a subclass of $\mcB_E$. If $\mcK$ is Fra\"iss\'e then there exists a unique (up to isomorphism) countable homogeneous $E$-structure whose age is equal to $\mcK$. This structure is called the \emph{Fra\"iss\'e limit} of $\mcK$.
\end{theorem}

To continue with our examples (which, to be rigorous, should be formulated in the traditional setup of Fra\"iss\'e classes rather than the one we are presenting here), $\Q$ is the Fra\"iss\'e limit of the class of finite linear orders, while the countable atomless Boolean algebra is the limit of the class of finite Boolean algebras.

To prove both the existence and uniqueness of the Fra\"iss\'e limit, one can use the following characterization.

\begin{prop}\label{p: Fraisse}
Let $\mcK$ be a subclass of $\mcB_E$, and assume that $\mcK$ is Fra\"iss\'e. Then a $E$-structure $M$ is isomorphic to the Fra\"iss\'e limit of $\mcK$ iff the age of $M$ coincides with $\mcK$ and $M$ satisfies the \emph{Fra\"iss\'e property}, namely: whenever $A \subseteq M$, $B \in \mcK$ and $\alpha \colon A \to B$ is an embedding, there exists an embedding $\beta \colon B \to M$ such that $\beta \alpha (a)=a$ for all $a \in A$.
\end{prop}

Let us give a bit more detail about this proposition and theorem \ref{t:Fraisse}. First, it is easy to see (using the same argument as in the proof of Proposition \ref{p:amalgamation}) that any homogeneous structure must satisfy the Fra\"iss\'e property. Next, a back-and-forth argument proves that two countable structures $M,N$ with the same age which both satisfy the Fra\"iss\'e property are isomorphic. Indeed, fix enumerations $(m_k)$ and $(n_k)$ of $M,N$ respectively. Denote by $A_k$ the (finite) substructure of $M$ generated by $m_0,\ldots,m_k$, and by $B_k$ the substructure of $N$ generated by $n_0,\ldots,n_k$. Using the Fra\"iss\'e property, one can build sequences $(A'_k), (B'_k)$ of finite substructures of $M,N$ respectively such that 
\begin{enumerate}
\item For each $k$, $A'_k \subseteq A'_{k+1}$, $B'_k \subseteq B'_{k+1}$ and $g_{k+1}$ extends $g_k$.
\item For each even $k$, $A_k$ is a subset of $A'_k$.
\item For each odd $k$, $B_k$ is a subset of $B'_k$.
\end{enumerate}
The second and third conditions above impose that $\bigcup_k A'_k=M$, $\bigcup_k B'_k=N$; and the map $\bigcup_k g_k$ is then the desired isomorphism between $M$ and $N$.

From the discussion above, we deduce the uniqueness of a Fra\"iss\'e limit up to isomorphism. Using an argument similar to the one we just discussed, it is also straightforward to prove that a structure which has the Fra\"iss\'e property is homogeneous. 

We have not explained how to construct the Fra\"iss\'e limit; the argument makes crucial use of the countability and amalgamation property of $\mcK$ to build a structure with age equal to $\mcK$ and which has the Fra\"iss\'e property. While not terribly difficult, the proof is a bit technical and we will not give it here. The curious reader may look up \cite{Hodges1993a}*{Theorem~7.1.2} for details; note that what Hodges calls ``weak homogeneity'' corresponds to we call here the Fra\"iss\'e property.

Given a dynamical simplex $K$ on a Cantor space $X$, we can see the clopen algebra of $X$, endowed with all the measures in $K$, as a countable $K$-structure. Crucially for us, this structure is homogeneous; thus, its age is a Fra\"iss\'e class. That is the content of the next lemma, which is itself an easy consequence of \cite{Ibarlucia2016}*{Proposition~2.7}.

\begin{lemma}\label{l:homogeneity}
Let $X$ be a Cantor space, $K$ be a dynamical simplex and $G$ denote the group of all homeomorphisms of $X$ which preserve all the elements of $K$.
Then for any two clopen partitions $(A_i)_{i=1,\ldots,n}$, $(B_i)_{i=1,\ldots,n}$ of $X$ such that $\mu(A_i)=\mu(B_i)$ for all $i \in \{1,\ldots,n\}$ and all $\mu \in K$, there exists $g \in G$ such that $g(A_i)=B_i$ for all $i \in \{1,\ldots,n\}$.
\end{lemma}

\begin{proof}
By Proposition~2.7 of \cite{Ibarlucia2016}, there exist elements $g_1,\ldots,g_n$ of $G$ such that $g_i(A_i)=B_i$ for all $i \in \{1,\ldots,n \}$. Then one obtains the desired $g \in G$ by setting $g(x)=g_i(x)$ whenever $x \in A_i$.
\end{proof}

Now we know that, associated to a dynamical simplex $K$, there is a natural Fra\"iss\'e class contained in $\mcB_K$.
We are now concerned with the other direction: building dynamical simplices from Fra\"iss\'e classes of $E$-structures (where $E$ is, as before, some nonempty fixed set). We will use the fact that the clopen algebra $\Clop(X)$ of a Cantor space $X$ is the unique (up to isomorphism) atomless countable Boolean algebra; and that the set of finitely additive probability measures on $\Clop(X)$ is naturally identified with the compact space $P(X)$ of Borel probability measures on $X$.

From now on, we only consider Fra\"iss\'e classes $\mcK$ contained in $\mcB_E$ which satisfy the nontriviality condition that, given any $A \in \mcK$, there is an embedding $\alpha \colon A \to B \in \mcK$ such that for all atoms $a \in A$ $\alpha(a)$ is not an atom of $B$. We call such classes \emph{diffused}.
Given any diffused Fra\"iss\'e class $\mcK$ contained in $\mcB_E$, the underlying Boolean algebra of the limit of $\mcK$ is a countable atomless Boolean algebra, which we see as the algebra of clopen sets of a Cantor space $X_{\mcK}$ by Stone duality; then for any $e \in E$ we see the measure $\mu_e$ as a probability measure on $X_\mcK$. We denote by $S(\mcK)$ the closed convex hull of $\{\mu_e\}_{e \in E}$ inside $P(X_\mcK)$.

\begin{prop}\label{Fraissesimplex}
Let $\mcK$ be a Fra\"iss\'e class of $E$-structures. Then $S(\mcK)$ is a dynamical simplex if and only if the following conditions are satisfied:
\begin{enumerate}
\item\label{one} For any $A \in \mcK$, and any nonzero $a \in A$, $\inf_e \mu_e(a) >0$.
\item\label{two} For any $A \in \mcK$, any $a \in A$ and any $\varepsilon >0$, there exists $B \in \mcK$ and $a_1,\ldots,a_n,a' \in B$ such that
$\bigvee a_i=a'$, $\mu_e(a_i) \le \varepsilon$ for all $i$ and all $e \in E$, and $\mu_e(a')=\mu_e(a)$ for all $e \in E$.
\item\label{three} For any $A,B \in \mcK$ and any $a \in A$, $b \in B$, if $\mu_e(a) \le \mu_e(b) - \varepsilon$ for all $e \in E$ and some $\varepsilon >0$ then there exists $C \in \mcK$ and $a',b' \in C$ such that $a'$ is contained in $b'$ and $\mu_e(a)=\mu_e(a')$, $\mu_e(b)=\mu_e(b')$ for all $e \in E$.
\end{enumerate}
\end{prop}

Implicit in the above Proposition is that, under its assumptions, $\mcK$ is automatically diffused. Indeed, pick $A \in \mcK$, and let $a$ be an atom of $A$. Find $B \in \mcK$ as in \eqref{two} for some $\varepsilon$ which is both $>0$ and strictly smaller than $\inf_e \mu_e(a)$. The $E$-structures $\{0,1,a,a^\complement\}$ and $\{0,1,a',a'^\complement\}$ are isomorphic; since $\mcK$ is Fra\"iss\'e we may amalgamate $A$ and $B$ over this common subalgebra to produce a new element $C_a \in \mcK$. The image of $a$ in $C_a$ under the associated embedding is not an atom (since $a'$ is not an atom in $B$).  Then we may inductively amalgamate the algebras $C_a$ as $a$ ranges over the atoms of $A$ to produce an element of $\mcK$ witnessing that $\mcK$ is diffused.

\begin{proof}[Proof of Proposition \ref{Fraissesimplex}]
The fact that \eqref{one} and \eqref{two} are necessary is well-known and easily checked (see for instance Proposition~2.5 of \cite{Ibarlucia2016}); \eqref{three} directly follows from the Glasner--Weiss property. 

Conversely, assume that all three conditions above are satisfied. To see that $S(\mcK)$ satisfies the Glasner--Weiss property, let $U,V$ be clopen subsets of the underlying Cantor space such that $\mu(U)< \mu(V)$ for all $\mu \in S(\mcK)$. Then $A=\{0,1,U,U^\complement\}$ and $B=\{0,1,V,V^\complement\}$, $a=U$, $b=V$ satisfy the conditions of \eqref{three}. Pick a $E$-structure $C$ witnessing that \eqref{three} holds; $C$ is in the age of $S(\mcK)$, so it is realized by some clopen partition. Using homogeneity we may assume that $b'=b$. Then the image of $a$ under the associated embedding is a clopen subset $W$ of $V$ such that $\mu(W)=\mu(U)$ for all $\mu \in S(\mcK)$.

Now fix $\nu \in S(\mcK)$. To see that $\nu$ has full support, pick some nonempty clopen $U$ in $X_\mcK$; then by \eqref{one} there exists $\varepsilon >0$ such that $\mu_e(U) \ge \varepsilon$ for all $e \in E$, hence also $\nu(U) \ge \varepsilon$. Finally, the combination of \eqref{two} and \eqref{three} (and homogeneity) implies that, given any clopen $U \subset X_\mcK$ and any $\varepsilon >0$ there exists a covering of $U$ by clopens $U_1,\ldots,U_n$ such that $\mu_e(U_i) \le \varepsilon$ for all $i$ and all $e \in E$, hence also $\nu(U_i) \le \varepsilon$ for all $i$, and this in turn implies that $\nu$ is atomless.
\end{proof}

\section{A Fra\"iss\'e theoretic proof of Downarowicz's theorem}
Given a Choquet simplex $Q$, we denote by $A(Q)$ the space of all real-valued continuous, convex, affine maps on $Q$. For $F \subset A(Q)$ we denote by $F^+$ the elements of $F$ taking positive values, and by $F_1$ the elements taking values in $[-1,1]$ (that is, the unit ball of $F$ for the sup-norm). Combining these notations, we denote by $F^+_{1}$ the elements of $F$ with values in the half-open interval $]0,1]$.

\begin{defn}
We say that a subset $F$ of $A(Q)$ satisfies the \emph{finite sum property} if, whenever $f,f_1,\ldots,f_n, g_1,\ldots,g_m$ are elements of $F^+$ such that $f= \sum_{i=1}^n f_i= \sum_{j=1}^m g_j$ there exist $h_{i,j} \in F^+$ satisfying
$$\forall j \in \{1,\ldots,n\} \ \sum_{k=1}^m h_{j,k} = f_j \ \text{ and } \forall k \in \{1,\ldots,m\} \  \sum_{j=1}^n h_{j,k}= g_k\ .$$ 
\end{defn}

It is an important fact that, when $Q$ is a Choquet simplex, $A(Q)$ satisfies the finite sum property, which is an equivalent formulation of the Riesz decomposition property of $A(Q)$ (see \cite{Lindenstrauss1964} or \cite{Edwards1965} and Theorem 2.5.4, Chapter I.2 of \cite{Lussky1981}; \cite{Phelps2001} and \cite{Lussky1981} are good references for the facts we use here about Choquet simplices).

We need a version of the finite sum property for maps with values in $]0,1]$, which is immediate and which we just state for the record.

\begin{lemma}\label{l:amalgamation}
Let $Q$ be a Choquet simplex. Assume that $F \subseteq A(Q)$ is a $\Q$-vector subspace which contains the constant function $1$ and has the finite sum property. Assume also that $f$, $(f_j)_{i =1,\ldots,n}$, $(g_k)_{k =1,\ldots,m}$ are elements of $F^+_{1}$ such that
$$ \sum_{j=1}^n f_j = f = \sum_{k=1}^m g_k \ .$$
Then there exist elements $h_{j,k}$ of $F^+_{1}$ such that
$$\forall j \in \{1,\ldots,n\} \ \sum_{k=1}^m h_{j,k} = f_j \ \text{ and } \forall k \in \{1,\ldots,m\} \  \sum_{j=1}^n h_{j,k}= g_k\ .$$
\end{lemma}

We now introduce the subclass $\mcK_Q$ of all finite $A \in \mcB_Q$ such that for all nonzero $a \in A$ the map $q \mapsto \mu_q(a)$ belongs to $A(Q)^+_{1}$. We should note again that this class is not countable, as there are too many possibilities for the values $\mu_q(a)$; later we will produce Fra\"iss\'e classes inside $\mcK_Q$. That is made possible by the following fact, which is key in our construction.

\begin{prop}\label{p:amalgamation2}
For any Choquet simplex $Q$, the class $\mcK_Q$ satisfies the amalgamation property.
\end{prop}

\begin{proof}
Let $A, B, C \in \mcK_Q$ and $\alpha \colon A \to B$, $\beta \colon A \to C$ be embeddings. We list the atoms of $A, B,C$ as $(a_i)_{i \in I}$, $(b_j)_{j \in J}$ and $(c_k)_{k \in K}$. For all $i \in I$ we let $J_i$ (resp. $K_i$) denote the set of all $j$ such that $b_j \in \alpha(a_i)$ (resp. $c_k \in \beta(a_i)$). We first define the underlying Boolean algebra of our amalgam, using the usual amalgamation procedure for Boolean algebras (as is done for instance in \cite{Kechris2007}): the atoms of $D$ are of the form $b_j \otimes c_k$ for each pair $(j,k)$ which belongs to $J_i \times K_i$ for some $i$.  For all $i$ and all $(j,k) \in J_i \times K_i$ we will set
$$\alpha'(b_j)= \bigvee_{k \in K_i} b_j \otimes c_k \text{ and } \beta'(c_k)= \bigvee_{j \in J_i} b_j \otimes c_k \ .$$ 
The one thing that requires some care is to define each $\mu_q(b_i \otimes c_j)$; and the constraint we have to satisfy is that, given $i, (j,k) \in J_i \times K_i$  and $q \in Q$ we must have
$$\mu_q(b_j) =\sum_{k \in K_i}  \mu_q(b_i \otimes c_j) \text{ and } \mu_q(c_k) =\sum_{j \in J_i} \mu_q(b_i \otimes c_j) $$
The fact that this is possible is guaranteed by the finite sum property and the fact that 
$$ \forall i \ \sum_{j \in J_i} \mu_q(b_j) = \mu_q(a_i)= \sum_{k \in K_i} \mu_q(c_k) \ . $$
Indeed, since each map $q \mapsto \mu_q(b_j), q \mapsto \mu_q(c_k)$ and $q \mapsto \mu_q(a_i)$ belongs to $A(Q)^+_{1}$, Lemma \ref{l:amalgamation} allows us to find elements $h_{j,k}$ of $A(Q)^+_{1}$ such that 
$$\forall q \in Q \ \forall j \in J_i \sum_{k \in K_i} h_{j,k}(q)= \mu_q(b_j) \text{ and } \forall k \in K_i \sum_{j \in J_i} h_{j,k}(q)= \mu_q(c_k)\ . $$
Setting $\mu_q(b_j \otimes c_k)= h_{j,k}(q)$, we are done.
\end{proof}

The next lemma is really just a simple version of the Lowenheim--Skolem theorem; at the request of the referee we give a proof.

\begin{lemma}\label{l:Lowenheim}
Let $Q$ be a Choquet simplex, and $F$ be a countable $\Q$-vector subspace of $A(Q)$. Then $F$ is countained in a countable $\Q$-vector subspace which satisfies the finite sum property.
\end{lemma}

\begin{proof}
We build an increasing sequence $E_n$ of countable $\Q$-vector subspaces of $A(Q)$ such that $E_0=F$ and, for any $k$, whenever $f,f_1,\ldots,f_n, g_1,\ldots,g_m$ are elements of $E_k^+$ such that $f= \sum_{i=1}^n f_i= \sum_{j=1}^m g_j$ there exist $h_{i,j} \in E_{k+1}^+$ satisfying
$$\forall j \in \{1,\ldots,n\} \ \sum_{k=1}^m h_{j,k} = f_j \ \text{ and } \forall k \in \{1,\ldots,m\} \  \sum_{j=1}^n h_{j,k}= g_k\ .$$ 
Then $\bigcup E_i$ will have the desired property. To see why the construction is possible, assume that $E_k$ has been built. 
The set $Z$ of all $(f_1,\ldots,f_n, g_1,\ldots,g_m)$ in $E_k^+$ such that $\sum_{i=1}^n f_i= \sum_{j=1}^m g_j$ is countable and, 
for any element $z=(f_1,\ldots,f_n, g_1,\ldots,g_m)$ of $Z$, we may find maps $h^z_{i,j} \in A(Q)^+_1$ such that
$$\forall j \in \{1,\ldots,n\} \ \sum_{k=1}^m h^z_{j,k} = f_j \ \text{ and } \forall k \in \{1,\ldots,m\} \  \sum_{j=1}^n h^z_{j,k}= g_j\ .$$
Then, let $E_{k+1}$ denote the $\Q$-vector subspace generated by $E_k$ and the (countable) set of all $h^z_{i,j}$ as $z$ ranges over $Z$.
\end{proof}

\begin{theorem}[Downarowicz \cite{Downarowicz1991}]\label{t:downarowicz}
Given a nonempty metrizable Choquet simplex $Q$, there exists a minimal homeomorphism of a Cantor space whose set of invariant measures is affinely homeomorphic to $Q$.
\end{theorem}

Note that Downarowicz obtains a more precise result: the minimal homeomorphism that he obtains is a dyadic Toeplitz flow, while here we do not have control over its dynamics. 

\begin{proof}
We need to build a dynamical simplex which is affinely homeomorphic to $Q$. Since $Q$ is metrizable, $A(Q)$ is separable, thus contains a countable dense $\Q$-vector subspace. Then we can apply Lemma \ref{l:Lowenheim} to produce a countable dense $\Q$-vector subspace $F$ of $A(Q)$ which contains the constant function $1$ and satisfies the finite sum property.

Consider the class $\mcL$ consisting of all $A \in \mcK_Q$ such that for all nonzero $a \in A$ the map $q \mapsto \mu_q(a)$ belongs to $F^+_{1}$. By the same argument as above, the finite sum property of $F$ ensures that $\mcL$ has the amalgamation property; clearly $\mcL$ also satisfies the hereditary property, and has only countably many elements up to isomorphism because $F$ is countable, hence $\mcL$ is a Fra\"iss\'e class. 

The fact that $\mcL$ satisfies condition \eqref{one} of Proposition \ref{Fraissesimplex} follows from the fact that elements of $A(Q)^+_1$ are continuous and take values in $]0,1]$, hence are bounded below by some strictly positive constant; condition \eqref{two} is easily deduced from the fact that for any $f \in F$ and any integer $N>0$ the function $\frac{1}{N} f$ also belongs to $F$. Finally, condition \eqref{three} is established by noticing that if $f,g \in F^+_1$ and $f(q)< g(q)$ for all $q \in Q$ then $g-f$ also belongs to $F^+_1$.

To sum up, we just built a dynamical simplex $S(\mcL)$. We now claim that the map $\Phi \colon q \mapsto \mu_q$ is a affine homeomorphism from $Q$ to $S(\mcL)$. 

Continuity of $\Phi$ is equivalent to the fact that for each $U \in \Clop(X_\mcL)$ the map $q \mapsto \mu_q(U)$ is continuous, which is built into our definition of $\mcL$ since this map belongs to $A(Q)$. Similarly, $\Phi$ is affine because each map $q \mapsto \mu_q(U)$ is affine.  This implies that $\Phi(Q)=\{\mu_q \colon q \in Q\}$ is compact and convex, so its closed convex hull $S(\mcL)$ is equal to it and $\Phi$ is onto. Finally, injectivity of $\Phi$ is ensured by the density of $F$ in $A(Q)$: if $q,q'$ are distinct elements of $Q$, there exists $f \in F^+_1$ such that $f(q) \ne f(q')$, and both $f(q),f(q')$ are different from $1$. Then by definition of $\mcL$ this yields $U \in \Clop(X_\mcL)$ such that $\mu_q(U)= f(q) \ne f(q')= \mu_{q'}(U)$. Thus $\Phi \colon Q \to S(\mcL)$ is an affine homeomorphism and we are done.
\end{proof}

\section{Speedup equivalence vs orbit equivalence}\label{s:speedups}

\begin{defn}
Two minimal homeomorphisms $\varphi, \psi$ of a Cantor space $X$ are \emph{orbit equivalent} if there is $g \in \Homeo(X)$ such that $g$ maps the $\varphi$-orbit of $x$ onto the $\psi$-orbit of $g(x)$ for all $x \in X$.
\end{defn}

\begin{defn}
For $\varphi \in \Homeo(X)$, we denote by $S(\varphi)$ the simplex of all $\varphi$-invariant Borel probability measures on $X$.
\end{defn}

We recall that for $\mu \in P(X)$ and $g \in \Homeo(X)$ one has $g_* \mu(A)=\mu(g^{-1}A)$ for all clopen $A$; and for $K \subset P(X)$ and $g \in \Homeo(X)$ we set $g_* K= \{g_* \mu \colon \mu \in K\}$.
The following theorem is a cornerstone of the study of orbit equivalence of minimal homeomorphisms of Cantor spaces.

\begin{theorem}[Giordano--Putnam--Skau \cite{Giordano1995}]
Whenever $\varphi$, $\psi$ are two minimal homeomorphisms of a Cantor space $X$, the following conditions are equivalent.
\begin{enumerate}
\item $\varphi$ and $\psi$ are orbit equivalent.
\item There exists $g \in \Homeo(X)$ such that $g_*S(\varphi)=S(\psi)$.
\end{enumerate}
\end{theorem}

Note that any $g$ witnessing that $\varphi$ and $\psi$ are orbit equivalent must be such that $g_*S(\varphi)= S(\psi)$; the converse is false, indeed it is possible that $S(\varphi)=S(\psi)$ but the orbits of $\varphi$ and $\psi$ do not coincide. The first example of (an even stronger instance of) this phenomenon is due to M. Boyle and appears as Appendix A of \cite{Ormes2000}. Another illuminating example is the following: on $X=\{0,1\}^{\N}$, consider the equivalence relation $E_0$ defined by 
$$\forall x,y \in \{0,1\}^{\N} \quad \left(x \ E_0 \ y\right) \Leftrightarrow \left( \exists N \ \forall n \ge N \ x_n=y_n \right)\ . $$
Consider also the dyadic odometer $\varphi$ (which corresponds to adding $1$ with right-carry); the associated equivalence relation $F$ is the same as $E_0$, except that the two classes $0^\infty$ and $1^\infty$ are merged. It follows from work of Giordano--Putnam--Skau that there exists a minimal homeomorphism $\psi$ of $\{0,1\}^\N$ whose associated equivalence relation is $F$ (J. Clemens gave an explicit construction of such a homeomorphism in unpublished work). Then $\varphi$, $\psi$ each only preserve one measure, namely, the $(\frac{1}{2},\frac{1}{2})$-Bernoulli measure on $\{0,1\}^\N$. In particular, $S(\varphi)=S(\psi)$. Yet, they do not have the same orbits, since as we mentioned before one $\varphi$-orbit splits into two $\psi$-orbits, while all the other orbits of $\varphi$ and $\psi$ are the same.

The Giordano--Putnam--Skau  theorem admits a one-sided analogue, due to Ash \cite{Ash2016}, which we proceed to describe.

\begin{defn}
Let $\varphi,\psi$ be minimal homeomorphisms. Say that $\varphi$ is a \emph{speedup} of $\psi$ if $\varphi$ is conjugate to a map of the form $x \mapsto \psi^{n_x}(x)$, with $n_x >0$ for all $x$.
\end{defn}

\begin{theorem}[Ash \cite{Ash2016}]
Whenever $\varphi$, $\psi$ are two minimal homeomorphisms of a Cantor space $X$, the following conditions are equivalent.
\begin{enumerate}
\item $\varphi$ is a speedup of $\psi$.
\item There exists $g \in \Homeo(X)$ such that $S(\psi) \subset g_*S(\varphi)$.
\end{enumerate}
\end{theorem}

Again, one of the two implications above is easy; the hard part is building the speedup from the assumption that one set of invariant measures is contained in the other. Ash showed that this can be done using Kakutani--Rokhlin partitions.

Note that, if $\varphi,\psi$ are minimal homeomorphisms such that each $\varphi$-orbit is contained in a $\psi$-orbit, then the set of $\varphi$-invariant measures contains the set of $\psi$-invariant measures; conversely, if the set of $\varphi$-invariant measures contains the set of $\psi$-invariant measures then it follows from Ash's result that one may conjugate $\varphi$ so that each $\varphi$-orbit is contained in a $\psi$-orbit (this is weaker than Ash's theorem cited above and admits a shorter proof). It is natural to wonder, as Ash did \cite{Ash2016}, whether speedup equivalence is actually the same notion as orbit equivalence, which also corresponds to asking whether there is a Schr\"oder-Bernstein property for equivalence relations induced by minimal homeomorphisms. As pointed out by Ash, one easily sees that such is the case when one of the homeomorphisms has a finite-dimensional simplex of invariant measures.

\begin{prop}[Ash \cite{Ash2016}]
Assume that $\varphi,\psi$ are minimal homeomorphisms and that $S(\varphi)$ has finitely many extreme points. If $\varphi$ and $\psi$ are speedup equivalent then they are orbit equivalent.
\end{prop}

The situation turns out to be very different for infinite-dimensional simplices.

\begin{defn}
We say that two dynamical simplices $K,L$ on a Cantor space $X$ are \emph{biembeddable} if there exists $g, h  \in \Homeo(X)$ such that $g_*K \subseteq L$ and $h_*L \subseteq K$.
\end{defn}

Thus, biembeddability is the translation of speedup equivalence at the level of dynamical simplices.

\begin{theorem}\label{t:badspeedups}
Assume that $Q,R$ are two nonempty metrizable Choquet simplices such that $Q$ affinely continuously embeds as a face of $R$ and $R$ affinely continuously embeds as a face of $Q$. Then there exist dynamical simplices $S(Q)$, $S(R)$ on a Cantor space $X$ such that $S(Q)$ is affinely homeomorphic to $Q$; $S(R)$ is affinely homeomorphich to $R$; $S(Q)$ and $S(R)$ are biembeddable.
\end{theorem}

In particular, given any $Q,R$ as above which are not affinely homeomorphic, and minimal homeomorphisms $\varphi_Q, \varphi_R$ with sets of invariant measures equal to $S(Q)$, $S(R)$ respectively, $\varphi_Q$ and $\varphi_R$ are speedup equivalent but not orbit equivalent. Such examples exist: for instance, let $Q$ be the Poulsen simplex, and $R$ the tensor product of $Q$ and $[0,1]$ (see \cite{Namioka1969}). Then $Q$ embeds as a closed face of $R$, and since any Choquet simplex is affinely homeomorphic to a closed face of the Poulsen simplex this is in particular true for $R$. But $Q$ and $R$ are not homeomorphic since the set of extreme points of $R$ is not dense in $R$.

There is some work to do before we can prove Theorem \ref{t:badspeedups}; we need an additional property of Choquet simplices.

\begin{prop}[Edwards \cite{Edwards1965}]\label{l:preserveinequalities}
Let $Q$ be a Choquet simplex, $R$ be a closed face of $Q$ and $f \in A(R)$. Assume that $f_1,f_2 \in A(Q)$ are such that 
${f_1}_{|R} \le f \le {f_2}_{|R}$. Then there exists $g \in A(Q)$ extending $f$ and such that $f_1 \le g \le f_2$. 
\end{prop}

\begin{defn}
Assume that $Q$ is a Choquet simplex, $R$ is a closed face of $Q$ and $F$ is a subset of $A(Q)$. 
We say that $F$ has the \emph{$(R,Q)$-extension property} if for any $f, f_1,\ldots,f_n \in F^+_1$  such  that 
$f_{|R} =  \sum_{i=1}^n {f_i}_{|R} $ there exist $g_1,\ldots,g_n \in F^+_1$ such that ${g_i}_{|R}={f_i}_{|R}$ for all $i \in \{1,\ldots,n\}$, and $f=\sum_{i=1}^n g_i$.
\end{defn}

The proof of the following lemma from Proposition \ref{l:preserveinequalities} is straightforward.
\begin{lemma}\label{l:preserveinequalities2}
Assume that $Q$ is a Choquet simplex and $R$ is a closed face of $Q$. Then $A(Q)$ has the $(R,Q)$-extension property.
\end{lemma}

\begin{proof}
By induction, it is enough to prove that the property holds for $n=2$. So, take $f, f_1, f_2 \in F^+_1$ such that $f_{|R}= {f_1}_{|R} + {f_2}_{|R}$. By compactness we have some $\varepsilon >0$ such that $\varepsilon \le {f_1}_{|R} \le {f}_{|R} - \varepsilon$, so we can apply Proposition \ref{l:preserveinequalities} and extend ${f_1}_{|R}$ to $g_1 \in A(Q)$ such that $\varepsilon \le g_1 \le f- \varepsilon$. These inequalities imply both that $g_1 \in A(Q)^+_1$ and that $g_2=f-g_1 \in A(Q)^+_1$.
\end{proof}

We move on towards proving Theorem \ref{t:badspeedups}.
Starting from $Q,R$ as in the statement of the theorem, we first assume that $R$ is a closed face of $Q$ and pick an affine continuous injection $g$ of $Q$ into itself such that $g(Q)$ is a closed face of $R$. 

Note that, given any $f \in A(Q)^+_1$, we may define $h \in A(g(Q))^+_1$ by setting $h(g(q))=f(q)$. Then, using Proposition 
\ref{l:preserveinequalities} we can extend $h$ to $A(Q)^+_1$. We just proved that, for any $f \in A(Q)^+_1$, there exists $h \in A(Q)^+_1$ such that $h \circ g=f$. Granting that, we may apply a closure argument similar to the one used in the proof of Lemma \ref{l:Lowenheim} and find $F \subset A(Q)$ such that:
\begin{itemize}
\item $F$ is a countable dense $\Q$-vector subspace which contains the constant function $1$.
\item $F$ satisfies the finite sum property.
\item $F$ satisfies both the $(R,Q)$-extension property and the $(g(Q),Q)$-extension property.
\item $\{f \circ g \colon f \in F^+_1\}= F^+_1$.
\end{itemize}

We then run the same construction as in the proof of Theorem \ref{t:downarowicz}, considering again the class $\mcL$ consisting of all $A \in \mcK_Q$ such that for all nonzero $a \in A$ the map $q \mapsto \mu_q(a)$ belongs to $F^+_{1}$. We denote by $M_Q$ the Fra\"iss\'e limit of this class, view its underlying Boolean algebra as the clopen algebra of a Cantor space $X_\mcL$ and again denote by $\Phi \colon Q \to S(\mcL)$ the affine homeomorphism $q \mapsto  \mu_q$.

We consider the classes $\mcL_R$  (resp. $\mcL_{g(Q)}$) obtained by taking restrictions of elements of $\mcL$ to $R$ (resp. $g(Q)$); that is, whenever $A_Q=(A,(\mu_q)_{q \in Q})$ is an element of $\mcL$, we obtain an element $A_R=(A,(\mu_r))_{r \in R}$ of $\mcL_R$, and all elements of $\mcL_R$ are obtained in this way (and, of course, similarly for $\mcL_{g(Q)}$).

\begin{lemma}
The classes $\mcL_R$ and $\mcL_{g(Q)}$  have the amalgamation property.
\end{lemma}

\begin{proof}
We only write down the argument for $\mcL_R$, the other one is identical.
As in the proof of Proposition \ref{p:amalgamation2}, showing that the amalgamation property holds reduces to proving that, whenever $(A,(\mu_r))_{r \in R}$ belongs to $\mcL_R$, $a \in A$ is a nonzero atom and $f_1,\ldots,f_n,g_1,\ldots,g_m \in F^+_1$ are such that 
$$\forall r \in R \ \mu_r(a)= \sum_{i=1}^n f_i(r)= \sum_{j=1}^m g_j(r) $$
there exist $h_{i,j} \in F^+_1$ such that 
$$ \forall i \ \sum_{j=1}^m {h_{i,j}}_{|R}= {f_i}_{|R} \text{ and } \forall j \ \sum_{i=1}^n {h_{i,j}}_{|R}= {g_j}_{|R} \ .$$
To prove that this is true, first apply the $(R,Q)$-extension property of $F$ to $f,f_1,\ldots f_n$ and $f,g_1,\ldots,g_m$, then use the fact that $F$ has the finite sum property.
\end{proof}

Let $M_R$, $M_{g(Q)}$ denote the structures $(\Clop(X_\mcL), (\mu_r)_{r \in R})$ and $(\Clop(X_\mcL),(\mu_{g(q)})_{q \in Q})$. 

\begin{lemma}
$M_R$ (resp. $M_{g(Q)}$) is the Fra\"iss\'e limit of $\mcL_R$ (resp. $\mcL_{g(Q)}$).
\end{lemma}

\begin{proof}
The $(R,Q)$-extension property of $F$ and the Fra\"iss\'e property of $M_Q$ combine to show that $M_R$ has the Fra\"iss\'e property. Indeed, let  $A$ be a clopen partition of $X_\mcL$, and pick an embedding $\alpha \colon A \to B \in \mcL_R$. For each atom $a$ of $A$, we let $(b^a_i)_{i \in I_a}$ denote the atoms of $B$ which are contained in $\alpha(a)$. The measures of each $b^a_i$ correspond to maps $f^a_i \in F^+_1$ such that $\sum_{i \in I_a} f^a_i(r)= \mu_a(r)$ for all $r \in R$. Then, the $(R,Q)$-extension property lets us pick maps $g^a_i \in F^+_1$ such that ${g^a_i}_{|R}={f^a_i}_{|R}$ and $\sum_{i \in I_a} g^a_i(q) = \mu_q(a)$ for all $q \in Q$.

Using the Fra\"iss\'e property of $M_Q$, we can then write each atom $a$ of $A$ as a union of clopen sets $U^a_i$ such that $\mu_q(U^a_i)=g^a_i(q)$ for all $q \in Q$; the map $\beta \colon b^a_i \mapsto U^a_i$ is then an embedding of $B$ inside $M_R$ such that $\beta \circ \alpha(a)=a$ for all $a \in A$.

The argument for $M_{g(Q)}$ is similar.

\end{proof}

Since $\mcL_R$, $\mcL_{g(Q)}$ both satisfy the conditions of Theorem \ref{Fraissesimplex}, we now have three dynamical simplices $S(\mcL_{g(Q)}) \subset S(\mcL_R) \subset S(\mcL_Q)$. Our last remaining task is to prove that there exists $h \in \Homeo(X)$ such that $h_*S(\mcL_Q)= S(\mcL_{g(Q)})$. Indeed, then the dynamical simplices $S(R)=S(\mcL_R)$, $S(Q)=S(\mcL_Q)$ satisfy the  desired conditions.

\begin{proof}[End of the proof of Theorem \ref{t:badspeedups}]

We define $N_Q= (\Clop(X_\mcL), (\nu_q)_{q \in Q}))$ by setting $\nu_q= \mu_{g(q)}$; $N_Q$ is a homogeneous $Q$-structure since $M_{g(Q)}$ is homogeneous, hence $N_Q$ is the Fra\"iss\'e limit of its age. The condition $\{f \circ g \colon f \in F^+_1\}= F^+_1$, combined with the $(g(Q),Q)$ extension property of $F$, says that the ages of $N_Q$ and $M_Q$ are the same; thus, $M_Q$ and $N_Q$ are isomorphic. Hence there is an automorphism $h$ of $\Clop(X_\mcL)$, equivalently a homeomorphism $h$ of $X_\mcL$, such that
$$\forall A \in \Clop(X_\mcL) \ \nu_q(h(A))= \mu_q(A)\ . $$
This amounts to stating that $h_* \mu_q = \nu_{g(q)}$ for all $q \in Q$, and in particular $h_*(S(\mcL_Q))= S(\mcL_{g(Q)})$.
\end{proof}

\bibliography{mybiblio}

\begin{bibdiv}
\begin{biblist}

\bib{Ash2016}{unpublished}{
      author={Ash, Drew},
       title={Topological speedups},
        date={2015},
        note={preprint available at \url{https://arxiv.org/abs/1605.08446}},
}

\bib{Downarowicz1991}{article}{
      author={Downarowicz, Tomasz},
       title={The {C}hoquet simplex of invariant measures for minimal flows},
        date={1991},
        ISSN={0021-2172},
     journal={Israel J. Math.},
      volume={74},
      number={2-3},
       pages={241\ndash 256},
         url={http://dx.doi.org.docelec.univ-lyon1.fr/10.1007/BF02775789},
      review={\MR{1135237}},
}

\bib{Edwards1965}{article}{
      author={Edwards, David~Albert},
       title={S\'eparation des fonctions r\'eelles d\'efinies sur un simplexe
  de {C}hoquet},
        date={1965},
     journal={C. R. Acad. Sci. Paris},
      volume={261},
       pages={2798\ndash 2800},
      review={\MR{0190720}},
}

\bib{Lussky1981}{book}{
      author={Fuchssteiner, Benno},
      author={Lusky, Wolfgang},
       title={Convex cones},
      series={North-Holland Mathematics Studies},
   publisher={North-Holland Publishing Co., Amsterdam-New York},
        date={1981},
      volume={56},
        ISBN={0-444-86290-0},
        note={Notas de Matem\'atica [Mathematical Notes], 82},
      review={\MR{640719}},
}

\bib{Giordano1995}{article}{
      author={Giordano, Thierry},
      author={Putnam, Ian~F.},
      author={Skau, Christian~F.},
       title={Topological orbit equivalence and {$C^*$}-crossed products},
        date={1995},
        ISSN={0075-4102},
     journal={J. Reine Angew. Math.},
      volume={469},
       pages={51\ndash 111},
      review={\MR{1363826 (97g:46085)}},
}

\bib{Glasner1995a}{article}{
      author={Glasner, Eli},
      author={Weiss, Benjamin},
       title={Weak orbit equivalence of {C}antor minimal systems},
        date={1995},
        ISSN={0129-167X},
     journal={Internat. J. Math.},
      volume={6},
      number={4},
       pages={559\ndash 579},
      review={\MR{MR1339645 (96g:46054)}},
}

\bib{Hodges1993a}{book}{
      author={Hodges, Wilfrid},
       title={Model theory},
      series={Encyclopedia of Mathematics and its Applications},
   publisher={Cambridge University Press},
     address={Cambridge},
        date={1993},
      volume={42},
        ISBN={0-521-30442-3},
      review={\MR{MR1221741 (94e:03002)}},
}

\bib{Ibarlucia2016}{unpublished}{
      author={Ibarluc\'ia, Tom\'as},
      author={Melleray, Julien},
       title={Dynamical simplices and minimal homeomorphisms},
        note={Proceedings of the American Mathematical Society, to appear},
}

\bib{Kechris2007}{article}{
      author={Kechris, Alexander~S.},
      author={Rosendal, Christian},
       title={Turbulence, amalgamation and generic automorphisms of homogeneous
  structures},
        date={2007},
     journal={Proc. Lond. Math. Soc.},
      volume={94},
      number={2},
       pages={302\ndash 350},
}

\bib{Lindenstrauss1964}{article}{
      author={Lindenstrauss, Joram},
       title={Extension of compact operators},
        date={1964},
        ISSN={0065-9266},
     journal={Mem. Amer. Math. Soc. No.},
      volume={48},
       pages={112},
      review={\MR{0179580}},
}

\bib{Namioka1969}{article}{
      author={Namioka, Isaac},
      author={Phelps, R.~R.},
       title={Tensor products of compact convex sets},
        date={1969},
        ISSN={0030-8730},
     journal={Pacific J. Math.},
      volume={31},
       pages={469\ndash 480},
  url={http://projecteuclid.org.docelec.univ-lyon1.fr/euclid.pjm/1102977882},
      review={\MR{0271689}},
}

\bib{Ormes2000}{article}{
      author={Ormes, Nicholas~S.},
       title={Real coboundaries for minimal {C}antor systems},
        date={2000},
        ISSN={0030-8730},
     journal={Pacific J. Math.},
      volume={195},
      number={2},
       pages={453\ndash 476},
         url={http://dx.doi.org/10.2140/pjm.2000.195.453},
      review={\MR{1782173}},
}

\bib{Phelps2001}{book}{
      author={Phelps, Robert~R.},
       title={Lectures on {C}hoquet's theorem},
     edition={Second},
      series={Lecture Notes in Mathematics},
   publisher={Springer-Verlag, Berlin},
        date={2001},
      volume={1757},
        ISBN={3-540-41834-2},
         url={http://dx.doi.org.docelec.univ-lyon1.fr/10.1007/b76887},
      review={\MR{1835574}},
}

\end{biblist}
\end{bibdiv}
\end{document}